\documentclass[11pt,reqno]{amsart}
\usepackage{amscd,amsmath,amsopn,amssymb,amsthm,multicol}
\usepackage[backref=page, breaklinks=true,colorlinks=true,linkcolor=blue,citecolor=blue,urlcolor=blue]{hyperref}

\usepackage{enumerate,stmaryrd}

\DeclareMathOperator{\Lie}{Lie}

\DeclareMathOperator{\ad}{ad}
\DeclareMathOperator{\Id}{Id}

\DeclareMathOperator{\diag}{diag}

\DeclareMathOperator{\Ad}{Ad}

\DeclareMathOperator{\Lin}{Lin}
\DeclareMathOperator{\Aut}{Aut}

\DeclareMathOperator{\Isom}{Isom}

\renewenvironment{proof}[1][Proof]{\textbf{#1.} }
{\ \rule{0.5em}{0.5em}}
\newtheorem{theorem}{Theorem}
\newtheorem{prop}{Proposition}
\newtheorem{lemma}{Lemma}

\theoremstyle{definition}
\newtheorem{definition}{Definition}
\newtheorem{remark}{Remark}

\begin{document}

\title
[On geodesic orbit nilmanifolds] {On geodesic orbit nilmanifolds}

\author{Yu.G.~Nikonorov}

\address{Southern Mathematical Institute of \newline
the Vladikavkaz Scientific Center of \newline
the Russian Academy of Sciences, \newline
Vladikavkaz, Vatutina st., 53, \newline
362027, Russia}
\email{nikonorov2006@mail.ru}

\begin{abstract}
The paper is devoted to the study of geodesic orbit Riemannian metrics on nilpotent Lie groups.
The main result is the construction of continuous families of pairwise non-isomorphic connected and simply connected
nilpotent Lie groups, every of which admits  geodesic orbit metrics. The minimum dimension of groups in the constructed families is $10$.

\vspace{2mm} \noindent 2020 Mathematical Subject Classification:
53C20, 53C25, 53C30, 17B30, 22E25.

\vspace{2mm} \noindent Key words and phrases: homogeneous Riemannian manifolds, geodesic orbit
spaces, naturally reductive spaces, nilmanifolds, nilpotent Lie groups, two-step nilpotent Lie algebra.
\end{abstract}

\maketitle

\section{Introduction and the main results}\label{sec_1}

A Riemannian manifold $(M,g)$ is called {\it a  manifold with
homogeneous geodesics or a geodesic orbit manifold} (shortly,  {\it GO-manifold}) if any
geodesic $\gamma $ of $M$ is an orbit of a 1-parameter subgroup of
the full isometry group of $(M,g)$. A Riemannian manifold $(M=G/H,g)$, where $H$ is a compact subgroup
of a Lie group $G$ and $g$ is a $G$-invariant Riemannian metric,
is called {\it a space with homogeneous geodesics} or {\it a geodesic orbit space}
(shortly,  {\it GO-space}) if any geodesic $\gamma $ of $M$ is an orbit of a
1-parameter subgroup of the group $G$.
Hence, a Riemannian manifold $(M,g)$ is  a geodesic orbit Riemannian manifold,
if it is a geodesic orbit space with respect to its full connected isometry group. This terminology was introduced in
\cite{KV} by O.~Kowalski and L.~Vanhecke, who initiated a systematic study on such spaces.
In the same paper, O.~Kowalski and L.~Vanhecke classified all GO-spaces
of dimension $\leq 6$. One can find many interesting results  about  GO-manifolds
and its subclasses in \cite{AA, AN, AV, BerNik08, BerNik09, CN2019, DuKoNi, Gor96, Nsp, S18, S20, Tam98, Tam99, Yakimova}, and in the references
therein.
\smallskip

It is clear that any geodesic orbit space is homogeneous.
All homogeneous spaces in this paper are assumed to be almost effective.
Let $(G/H, g)$ be a homogeneous Riemannian space. It is well known that there is an $\Ad(H)$-invariant decomposition (that is not unique in general)
\begin{equation}\label{reductivedecomposition}
\mathfrak{g}=\mathfrak{h}\oplus \mathfrak{p},
\end{equation}
where $\mathfrak{g}={\rm Lie }(G)$ and $\mathfrak{h}={\rm Lie}(H)$.
The Riemannian metric $g$ is $G$-invariant and is determined
by an $\Ad(H)$-invariant Euclidean metric $g = (\cdot,\cdot)$ on
the space $\mathfrak{p}$ which is identified with the tangent
space $T_oM$ at the initial point $o = eH$. By $[\cdot, \cdot]$ we denote the Lie bracket in $\mathfrak{g}$, and by
$[\cdot, \cdot]_{\mathfrak{p}}$ its $\mathfrak{p}$-component according to (\ref{reductivedecomposition}). The following is
(in the above terms) a well-known criteria of GO-spaces, see other details and useful facts in \cite{BerNik20, Nik2017}.

\begin{lemma}[\cite{KV}]\label{GO-criterion}
A homogeneous Riemannian space   $(G/H,g)$ with the reductive
decomposition  {\rm(\ref{reductivedecomposition})} is a GO-space if and
only if  for any $X \in \mathfrak{p}$ there is $Z \in \mathfrak{h}$ such that
$$
([X+Z,Y]_{\mathfrak{p}},X) =0  \text{ for all } Y\in \mathfrak{p}.
$$
\end{lemma}

It is clear that the property to be geodesic orbit is related to classes of locally isomorphic homogeneous spaces due to this lemma.

The metric $g$ is called \emph{naturally reductive} if an $\Ad(H)$-invariant complement $\mathfrak{p}$
can be chosen in such a way that $([X,Y]_{\mathfrak{p}},X) = 0$ for all $X,Y \in \mathfrak{p}$.
In this case, we say that
the (naturally reductive) metric $g$ \emph{is generated by the pair} $(\mathfrak{p}, (\cdot ,\cdot ))$.
It immediately follows that any naturally reductive space is a geodesic orbit space; the converse is false when
$\dim (M) \ge 6$ \cite{KV}. It should be noted that the property of being naturally reductive depends on the choice of the group $G$
(the choice of the presentation $M=G/H$);
both enlarging and reducing $G$ may result in gaining or losing this property, see details e.g. in \cite{Nik2017}.
Every isotropy irreducible Riemannian space is naturally reductive, and hence geodesic orbit, see e.g.~\cite{Bes}.

The class of (Riemannian) geodesic orbit spaces includes (but is not limited to) symmetric spaces, weakly symmetric spaces \cite{AV, BKV, Ngu2, W1, Yakimova, Zil96},
naturally reductive spaces \cite{AFF, DZ, Gor85, KV85, Storm19, Storm20}, normal and
generalised normal homogeneous ($\delta$-homogeneous) spaces \cite{BerNik08, BerNik12, BerNik20},
and Clifford~--~Wolf homogeneous manifolds \cite{BerNik09,BerNik20}.
For the current state of knowledge in the theory of geodesic orbit spaces and manifolds we refer the reader to the book \cite{BerNik20}, the papers
\cite{AN, Arv17, Gor96, Nik2017, Storm20, CNN2023},
and the references therein.

It should be noted that GO property is a very general geometric phenomenon: it is extensively studied in Riemannian, Lorentzian and general
pseudo-Riemannian settings (see \cite{Bar, CWZ2022, NW23, WC22}), in Finsler geometry (see recent papers \cite{Du2, XDY, YD} and the references therein),
in affine geometry \cite{Du1}, and even for finite
metric spaces \cite{BerNik19}. In all these cases, is not hard to see that the GO property implies homogeneity, but is much stronger.

\medskip

There is no hope to obtain a complete classification of all Riemannian geodesic orbit spaces. Partial classifications
are possible only for special types of geodesic orbit metrics (for instance,
Clifford~--~Wolf homogeneous metrics \cite{BerNik09}) or for small dimensions (for $\dim (M) \le 6$ see \cite{KV} and references therein).
\medskip

In this paper, we deal with Riemannian geodesic orbit metrics on nilpotent Lie groups, that can be studied using nilpotent Lie algebras
equipped  with suitable inner products.
It is known that there are finite numbers of isomorphism classes of complex or real nilpotent Lie algebra in $\dim \leq 6$.
On the other hand there are six $1$-parameter families of nilpotent Lie algebras of
dimension $7$, pairwise not isomorphic \cite{GoKham1996}.

Two-step nilpotent (metabelian) Lie algebras form the first non-trivial
subclass of nilpotent algebras. However even the classification of these special nilpotent Lie algebras
is a rather complicated problem. This problem is completely solved in the case of $1$-dimensional or $2$-dimensional center \cite{LT99}.
Known results on small-dimensional two-step nilpotent Lie algebras
(in particular, the classification of complex two-step nilpotent Lie algebra in $\dim \leq 9$) can be found in \cite{GaTi99, IKP22}.
It should be noted that there are several continuous families of pairwise non-isomorphic two-step nilpotent Lie algebras in dimension $9$.
\smallskip

The main goal of this paper is to prove that the set of nilpotent groups admitting
Riemannian geodesic orbit metrics is quite extensive.
To do this, we will construct new examples of geodesic orbit metrics.
The first main result of this paper is the following

\begin{theorem}\label{tm_1}
There is a $1$-parameter family of pairwise non-isomorphic connected and simply connected $10$-dimensional nilpotent Lie groups $N_t$ such that every of them
admits $3$-parameter family of Riemannian geodesic orbit metrics.
\end{theorem}

Our second main result is the following generalization of the previous theorem.

\begin{theorem}\label{tm_2}
For any $k\geq 1$, there is a $k$-parameter family of pairwise non-isomorphic connected and simply connected  nilpotent Lie groups
$N_{t_1,t_2,\dots,t_k}$ of dimension $4k+6$, such that every of them
admits $3$-parameter family of Riemannian geodesic orbit metrics.
\end{theorem}

In addition, we pay attention to a special class of GO-nilmanifolds, namely
GO-nilmani\-folds of the centralizer type.
We establish some properties of such GO-nilmanifolds
that allow us to hope to obtain their classification at least for small dimensions.
\smallskip

The paper is organized as follows. In Section \ref{sec_2}, we recall important results on
Riemannian geodesic orbit metrics on nilpotent Lie groups. The main role here is played by C.~Gordon's results on the structure of geodesic orbital nilmanifolds
and on the description of GO-metrics on nilpotent Lie groups.
We consider some natural example of GO-nilmanifolds in Section \ref{sec_3}.
In Section \ref{sec_4}, we discuss GO-nilmanifolds of the centralizer type, some special class of GO-nilmanifolds that contains (in particular)
all two-step nilpotent GO-nilmanifolds with $2$-dimensional center.
Finally, we prove the main results in Section \ref{sec_5}.

\section{Riemannian geodesic orbit metrics on nilpotent Lie groups}\label{sec_2}

We discuss some properties of GO-nilmanifolds. The foundations of the corresponding theory were developed by C.~Gordon in~\cite{Gor96}.
We recall some important facts. In what follows we consider only connected and simply connected nilpotent Lie group $N$ supplied with some left-invariant
Riemannian metric $g$, and we call $(N,g)$ a nilmanifold.

The book \cite{GoKham1996} can be cited as a standard source on the theory of nilpotent groups and Lie algebras.
Note that the class of nilpotent Lie algebras is very wide and there
is no hope of obtaining a reasonable classification of them in an arbitrary dimension. Nevertheless, the classification of nilpotent
Lie algebras of small dimensions is known.
The classification of complex nilpotent Lie algebras of small dimension has a long history, yet only for dimension $\leq 7$ has it been completed, see
e.~g. \cite{MiJi19} for a survey.
Recall that the classification of complex two-step nilpotent (metabelian, in other terms) Lie algebras of dimension $\leq 9$ is obtained in \cite{GaTi99}.
See also \cite{YD13} and \cite{IKP22}. Important structure and (partial) classification results on 2-step nilpotent Lie algebra could be found in the following papers
by P.~Eberlein:
\cite{Eber1, Eber2, Eber3}.

{\bf In what follows, we consider only real nilpotent Lie algebras $\mathfrak{n}$}.
Recall that the corresponding Lie groups $N$ are assumed connected and simply connected. This imply that $N$ is diffeomorphic to a Euclidean space
(a detailed description of GO-manifolds
diffeomorphic to Euclidean spaces is obtained in \cite{GorNik2018}).
\smallskip

It is known that the full connected isometry group $G=\Isom(N,g)$ of a given nilmanifold
$(N,g)$ is such that $N$ is the nilradical of $G$, in particular, $N$ is a normal subgroup in $G$~\cite{Wil}.
We denote by $H$ the isotropy subgroup of $G$ at the unit element $e\in N$.

For $G/H$ as above, the Lie algebra $\mathfrak{n}={\rm Lie}(N)$ is an ideal in $\mathfrak{g}={\rm Lie}(G)$, hence we can write
\begin{equation}\label{reductivedecomposition1}
\mathfrak{g}=\mathfrak{h}\oplus \mathfrak{n},
\end{equation}
vector space direct sum, which is $\Ad_G(H)$-invariant.
The Riemannian metric $g$ corresponds to an $\Ad_G(H)$-invariant inner product $g_{eH}=(\cdot,\cdot)$ on $\mathfrak{n}$.
Let $O(\mathfrak{n}, (\cdot,\cdot))$ be the group of orthogonal maps on the metric Lie algebra $(\mathfrak{n}, (\cdot,\cdot))$
and $D(\mathfrak{n})$ the space of skew-symmetric derivations of the metric Lie algebra $(\mathfrak{n},(\cdot,\cdot))$.

If $\Phi$ is an automorphism of $G$ that normalizes $H$, then $\Phi$ induces a well-defined diffeomorphism $\overline{\Phi}$ of $G/H=N$ by
$\overline{\Phi}(aH)=\Phi(a)H$. Let us consider
$$
\Aut_{\rm{orth}}(G/H=N,g)=\{\overline{\Phi}: \Phi\in\Aut(G), \,\Phi(H)=H, \,\mbox{and}\,\overline{\Phi}_{*eH}\in O(\mathfrak{n}, (\cdot,\cdot))\}
$$
where $\Aut(G)$ is the authomormism group of the Lie group $G$
and $\overline{\Phi}_*$ is the differential of $\overline{\Phi}$.
We have

\begin{lemma}[E.N.~Wilson \cite{Wil}]\label{lem.wilson} Let $(N,g)$ be a Riemannian nilmanifold and $(\cdot,\cdot)$ the associated inner product on the Lie algebra $\mathfrak{n}$.
Then $\Isom(N,g)=N\rtimes H$, where $H=\Aut_{\rm{orth}}(N,g)$.
Thus the full isometry algebra of $(N,g)$ is the semi-direct sum
$\mathfrak{n}\rtimes \mathfrak{h}$, where~$\mathfrak{h}$
is the space $D(\mathfrak{n})$ of skew-symmetric derivations of  $(\mathfrak{n},(\cdot,\cdot))$.
\end{lemma}

In particular, if Riemannian nilmanifolds $(N_1,g_1)$ and $(N_2,g_2)$ are isometric to each other, then the Lie group $N_1$ is isomorphic to $N_2$, as well as
their Lie algebras are isomorphic to each other.
\smallskip

Lemma \ref{lem.wilson} implies that the full isometry algebra of $(N,g)$ is the semi-direct sum $\mathfrak{n}\rtimes \mathfrak{h}$, where~$\mathfrak{h}$
is the space $D(\mathfrak{n})$ of skew-symmetric derivations of the metric Lie algebra $(\mathfrak{n},(\cdot,\cdot))$ \cite{Wil, GorWil1985}.
Let us recall the following important result.

\begin{prop}[C.~Gordon \cite{Gor96}]\label{gonil1}
If $(N,g)$ is geodesic orbit Riemannian manifold, then the Lie algebra $\mathfrak{n}= {\rm Lie}(N)$ is either commutative or two-step nilpotent.
\end{prop}

In the case when $\mathfrak{n}$ is commutative, $(N,g)$ is Euclidean space. Hence, {\bf in what follow we suppose that $\mathfrak{n}$ is two-step nilpotent}.
\medskip

Now we recall one helpful method to represent any two-step nilpotent metric Lie algebra.
Let $\mathfrak{n}$ be a two-step nilpotent Lie algebra with an inner product $(\cdot,\cdot)$.
Denote by $\mathfrak{z}$ the center of $\mathfrak{n}$ and by $\mathfrak{v}$ the $(\cdot,\cdot)$-orthogonal complement to $\mathfrak{z}$ in $\mathfrak{n}$.
It is clear that $[\mathfrak{n}, \mathfrak{n}]= [\mathfrak{v}, \mathfrak{v}] \subset \mathfrak{z}$.
We denote by $\mathfrak{so}(\mathfrak{z})$ and $\mathfrak{so}(\mathfrak{v})$ the algebras of skew symmetric transformations of $(\mathfrak{z}, (\cdot,\cdot))$ and
$(\mathfrak{v}, (\cdot,\cdot))$ respectively.
It is easy to see that any $D\in D(\mathfrak{n})=\mathfrak{h}$ saves both $\mathfrak{z}$  and  $\mathfrak{v}$.
\smallskip

For any $Z\in \mathfrak{z}$, we consider the operator
\begin{equation}\label{eq_j_z_1}
J_Z:\mathfrak{v} \rightarrow \mathfrak{v}, \mbox{\,\,\,\, such that\,\,\,\,}
(J_Z(X),Y)=([X,Y],Z), \quad X,Y\in \mathfrak{v}.
\end{equation}
It is clear that $J_Z$ are skew-symmetric and $J_Z(Y)=(\ad Y)'(Z)$, where $(\ad Y)'$ is adjoint to $\ad Y$ with respect to $(\cdot,\cdot)$.
The map $J:Z \rightarrow J_Z$ is obviously linear.

\smallskip
Recall the following important result.

\begin{prop}[C.~Gordon \cite{Gor96}]\label{gonil1n}
In the above notations, $(N,g)$ is geodesic orbit Riemannian manifold if and only if for any $X\in \mathfrak{z}$ and
$Y\in \mathfrak{v}$ there is $D\in D(\mathfrak{n})$ such that
$[D,X]=D(X)=0$, $[D,Y]=D(Y)=J_X(Y)$.
\end{prop}

It is clear that
$J_Z\equiv 0$ for $Z\in \mathfrak{z}$ if and only if $Z$ is orthogonal to $[\mathfrak{v},\mathfrak{v}]\subset \mathfrak{z}$
(recall that $(J_Z(X),Y)=([X,Y],Z)$ for $X,Y \in \mathfrak{v}$).
Therefore, $J:\mathfrak{z} \rightarrow \mathfrak{so}(\mathfrak{v})$ is an injective map for any two-step nilpotent metric Lie algebra with
$[\mathfrak{n},\mathfrak{n}]=\mathfrak{z}$, because
$\mathfrak{z}=[\mathfrak{n},\mathfrak{n}]=[\mathfrak{v},\mathfrak{v}]$.
It should be noted that the latter condition is not too restrictive.

\begin{lemma}\label{le_nonsing}
Let $\mathfrak{n}$ be a two-step nilpotent Lie algebra with the center $\mathfrak{z}$ distinct from $[\mathfrak{n},\mathfrak{n}]$
Then $\mathfrak{n}$, supplied with an inner product $(\cdot,\cdot)$, generates a Riemannian nilmanifold $(N,g)$ that is a direct metric product of some Euclidean space
and some nilmanifold $(N_1,g_1)$.
Here $N_1$ is a Lie subgroup of $N$, $g_1$ is the restriction of $g$ to $N_1$,
and the Lie algebra $\mathfrak{n}_1$ of $N_1$ is a semi-direct sum of
$[\mathfrak{n},\mathfrak{n}]$ and $\mathfrak{v}$, where $\mathfrak{v}$ is an $(\cdot,\cdot)$-orthogonal complement to $\mathfrak{z}$ in $\mathfrak{n}$.
Moreover, the center of $\mathfrak{n}_1$ is $[\mathfrak{n}_1,\mathfrak{n}_1]=[\mathfrak{n},\mathfrak{n}]=[\mathfrak{v},\mathfrak{v}]$.
\end{lemma}

\begin{proof}
We consider the Riemannian connection $\nabla$ associated with a
Riemannian metric~$g$. This connection assigns to each pair of smooth
vector fields $X$ and $Y$ a smooth vector field $\nabla_XY$ called the covariant
derivative of $Y$ in the direction $X$.

If $X, Y, Z$ are all left invariant vector fields on a Lie group $N$ with a left invariant Riemannian metric $g$, then
we have the following formula:
$$
(\nabla_X Z,Y) = \frac{1}{2} \Bigl(([X,Z],Y) - ([Z,Y],X) + ([Y, X], Z)\Bigr),
$$
see e.~g.  (5.3) in \cite{Milnor}.
If $Z \in \mathfrak{z}$ is such that $(Z,[\mathfrak{n},\mathfrak{n}])=0$, then $(\nabla_X Z,Y)=0$ for all $X,Z \in \mathfrak{n}$, hence, $\nabla_X Z=0$ for
any $X \in \mathfrak{n}$.
Therefore, the distribution on $(N,g)$, generated by $\{Z\in \mathfrak{z}\,|\, (Z,[\mathfrak{n},\mathfrak{n}])=0\}$, is parallel (with respect to $\nabla$) and it
determines a flat (Euclidean) factor in $(N,g)$ (see e.~g. Theorem~10.43 in \cite{Bes}).
In other words, $(N,g)$ is a direct metric product of some Euclidean space and some nilmanifold $(N_1,g_1)$,
where $N_1$ is a subgroup of $N$ with the Lie subalgebra $[\mathfrak{n},\mathfrak{n}]$, while $g_1$ is the restriction of $g$ to $N_1$.
Now it is clear that $[\mathfrak{n}_1,\mathfrak{n}_1]=[\mathfrak{n},\mathfrak{n}]=[\mathfrak{v},\mathfrak{v}]$ is the center of $\mathfrak{n}_1$.
\end{proof}
\medskip

{\bf In what follows, we suppose that} $[\mathfrak{n},\mathfrak{n}]=\mathfrak{z}$, $m:=\dim (\mathfrak{z})$,
$n:=\dim (\mathfrak{v})=\dim (\mathfrak{n})-\dim (\mathfrak{z})$.
In particular, the linear map $J:=\mathfrak{z} \mapsto J_Z$ is injective, $\mathcal{V}=\{J_Z\,|\, Z\in \mathfrak{z}\}$ is $m$-dimensional linear subspace
in $\mathfrak{so}(\mathfrak{v})$.
\bigskip

If $\varphi:\mathfrak{h} \to \mathfrak{so}(\mathfrak{v})$ is the restriction of isotropy representation to $\mathfrak{v}$, we may reformulate the condition
of Proposition \ref{gonil1n} as follows.
We know that $\mathcal{V}=J(\mathfrak{z})$ is a linear subspace in $\mathfrak{so}(\mathfrak{v})$.
Further, for every $X\in \mathfrak{h}$ and $Z\in \mathfrak{z}$ we get
$J_{[X,Z]}=[\varphi(X),J_Z]$ (it easily follows from the condition on $X$ to be skew-symmetric derivation), hence,
the subspace $\mathcal{V}=J(\mathfrak{z})$ is normalized by the subalgebra $\mathcal{N}:=\varphi(\mathfrak{h})$ in
$\mathfrak{so}(\mathfrak{v})$. The equality $J_{[X,Z]}=[\varphi(X),J_Z]$ implies that
the representation $\varphi:\mathfrak{h} \to \mathfrak{so}(\mathfrak{v})$ is faithful
(otherwise, some non-trivial $X\in \mathfrak{h}$ acts trivially both on $\mathfrak{v}$
and on $\mathfrak{z}$, hence, on $\mathfrak{n}$).
Therefore, we have

\begin{enumerate}[\quad a)]
\item a Lie subalgebra $\mathcal{N}\subset \mathfrak{so}(\mathfrak{v})$ (acted on $\mathfrak{v}$) and
\item an $\ad(\mathcal{N})$-invariant
module $\mathcal{V}$ in $\mathfrak{so}(\mathfrak{v})$,
\end{enumerate}
such that
for every $Y\in \mathfrak{v}$ and $Z\in \mathcal{V}$  there is $X \in \mathcal{N}$ with the following properties:
$[X,Z]=0$ and $X(Y)=Z(Y)$.
\smallskip

Since $\dim (\mathfrak{v})=n$, we naturally identify $\mathfrak{so}(\mathfrak{v})$ with $\mathfrak{so}(n)$.
\smallskip

\begin{definition}[ \cite{Gor96}]\label{de_tnc}
Let $\mathcal{V}$ be a linear subspace of $\mathfrak{so}(n)$ and $\mathcal{N}$ the normalizer of $\mathcal{V}$ in $\mathfrak{so}(n)$.
We say that $\mathcal{V}$
satisfies {\it the transitive normalizer condition} if
for every $Y\in \mathbb{R}^n$ and every $Z\in \mathcal{V}$  there is some $X \in \mathcal{N}$ such that
$[X,Z]=0$ and $X(Y)=Z(Y)$.
\end{definition}

\begin{prop}[C.~Gordon \cite{Gor96}]\label{gonil2}
Let $\mathcal{V}$ be a linear subspace of $\mathfrak{so}(n)$ with the normalizer $\mathcal{N}\subset \mathfrak{so}(n)$.
Suppose that $\mathcal{V}$
satisfies the transitive normalizer condition.
Then the metric Lie algebra $(\mathcal{V}\rtimes \mathbb{R}^n, (\cdot,\cdot)_1+(\cdot,\cdot)_2)$
defines a geodesic orbit nilmanifold, where
$(\cdot,\cdot)_1$ is any $\ad(\mathcal{N})$-invariant inner product on $\mathcal{V}$,
$(\cdot,\cdot)_2$ is the standard inner product in $\mathbb{R}^n$,
$[X,Y]=0$ if $X\in \mathcal{V}$ and $Y\in \mathcal{V}$ or $Y\in \mathbb{R}^n$,
and $([X,Y],Z)_1=(Z(X),Y)_2$ for all $X,Y \in \mathbb{R}^n$ and $Z\in \mathcal{V}$.
In particular, $\mathcal{V}$ is the derived algebra of $\mathcal{V}\rtimes \mathbb{R}^n$ and, moreover, if for any $Y\in \mathbb{R}^n$ there is $Z\in \mathcal{V}$
such that $0\neq Z(Y) \in \mathbb{R}^n$, then $\mathcal{V}$ is the center of $\mathcal{V}\rtimes \mathbb{R}^n$.
\end{prop}

\begin{remark}\label{re_sub_1}
Suppose, that a linear subspace $\mathcal{V}$ of $\mathfrak{so}(n)$ with the normalizer $\mathcal{N}\subset \mathfrak{so}(n)$
satisfies the transitive normalizer condition.
It is possible that there is a Lie subalgebra  $\mathcal{N}'\subset  \mathcal{N}$ such that
for every $Y\in \mathbb{R}^n$ and every $Z\in \mathcal{V}$  there is some $X \in \mathcal{N}'$ such that
$[X,Z]=0$ and $X(Y)=Z(Y)$.
This means that even the subalgebra $\mathcal{N}'$ together with $\mathcal{V}$ generate a geodesic orbit nilmanifold.
Moreover, since the condition on
$(\cdot,\cdot)_1$ to be $\ad(\mathcal{N}')$-invariant is weaker than the condition to be $\ad(\mathcal{N})$-invariant, we can get more GO-metric using
less extensive subgroup $\mathcal{N}'$.
\end{remark}

\begin{definition}\label{de_tncwrt}
We will say that $\mathcal{V}$
satisfies {\it the transitive normalizer condition with respect to $\mathcal{N}'$} if
$\mathcal{N}'\subset  \mathcal{N}$ is as in Remark \ref{re_sub_1}.
\end{definition}
\smallskip

One obvious possibility to choose $\mathcal{V}$, satisfied the transitive normalizer condition, is the following:
$\mathcal{V}$ is a Lie subalgebra of $\mathfrak{so}(n)$. The following result is valid (see Section 2 in~\cite{Gor96}).

\begin{prop}\label{pr_natred_1}
Suppose that $\mathcal{V}$ is a Lie subalgebra of  $\mathfrak{so}(n)$ {\rm(}in particular,  $\dim(\mathcal{V})=1${\rm)}.
Then $\mathcal{V}\subset \mathcal{N}$ and we can take $\mathcal{N}'=\mathcal{V}$ in the notation of Remark \ref{re_sub_1}.
Any corresponding GO-nilmanifold $(N,g)$ {\rm(}that depends on $\ad(\mathcal{V})$-invariant inner product on $\mathcal{V}${\rm)} is naturally reductive.
On the other hand, if a subspace $\mathcal{V}\subset \mathfrak{so}(n)$ determined a naturally reductive GO-manifold, then $\mathcal{V}$
is a subalgebra of  $\mathfrak{so}(n)$.
\end{prop}

\begin{proof}
Since $[\mathcal{V}, \mathcal{V}]\subset \mathcal{V}$ and $\mathcal{N}=\{U\in \mathfrak{so}(n) \,|\, [U,\mathcal{V}]\subset \mathcal{V}\}$, then
$\mathcal{V} \subset \mathcal{N}$. All other assertions easily follows from Proposition \ref{gonil2}, Remark \ref{re_sub_1}, and Theorem 2.8 in \cite{Gor96}.
\end{proof}

\begin{remark}\label{re_iso_2}
If $\dim (\mathcal{V})=1$, then the structure of the corresponding nilpotent Lie algebra depend of one operator
$J_Z$, where $Z \in \mathcal{V}$ and has norm $1$.
The simplest non-trivial example is $J_Z=\diag \left(J,J,\cdots,J\right)$, where we have $n$ numbers of blocks
$J=\left(%
\begin{array}{cc}
  0 & 1 \\
  -1 & 0 \\
\end{array}%
\right)$, corresponds to a simply-connected Heisenberg
group of dimension $2n + 1$ for every
$n \geq 1$.
In the general case, $J_Z$ could be any skew-symmetric, then the corresponding
GO-nilmanifolds constructed from operators $J_{Z_1}$ and $J_{Z_2}$ with distinct sets of eigenvalues (counted with multiplicities) are not isometric each to other.
\end{remark}

\begin{remark}\label{re_iso_2.5}
It is possible that $\mathcal{V}$
is a Lie  subalgebra of  $\mathfrak{so}(n)$, but its normalizer $N$ contains a Lie subalgebra $\mathcal{V}\oplus \mathcal{N}'$ such that
$\mathcal{V}$ satisfies the transitive normalizer condition with respect to $\mathcal{N}'$. In this case $\mathcal{V}$
can be supplied with an arbitrary inner product, since any such inner product is $\ad(\mathcal{N}')$-invariant ($[\mathcal{V},\mathcal{N}']=0$).
If $\mathcal{V}$ is not abelian, then there is an inner product on $\mathcal{V}$, that is not $\ad(\mathcal{V})$-invariant.
Therefore, the corresponding GO-nilmanifold is not naturally reductive.
The simplest example is $\mathcal{V}\oplus \mathcal{N}'=\mathfrak{so}(3)\oplus \mathfrak{so}(3)=\mathfrak{so}(4)$, see
details and other examples see Section \ref{sec_3}.
\end{remark}

In what follows, we will work with GO-nilmanifolds that are not naturally reductive.
We know that the relation $[\mathcal{V}, \mathcal{V}]\not \subset \mathcal{V}$ is sufficient for this.
\smallskip

Let $\mathfrak{n}$ be a two-step nilpotent Lie algebra with the center $\mathfrak{z}$. Then  $\mathfrak{n}$ is called
{\it non-singular} (often called also {\it regular} or {\it fat}) if the operator $\ad(X): \mathfrak{n} \rightarrow \mathfrak{z}$
is surjective for all $X \in \mathfrak{n} \setminus \mathfrak{z}$.

It is obvious that $[\mathfrak{n},\mathfrak{n}]=\mathfrak{z}$ for any two-step nilpotent non-singular Lie algebra.
\smallskip

Let $\mathfrak{n}$ be a two-step nilpotent non-singular Lie algebra supplied with an inner product $(\cdot,\cdot)$, $\mathfrak{z}$ is the center of $\mathfrak{n}$
and $\mathfrak{v}$ is an $(\cdot,\cdot)$-orthogonal complement to $\mathfrak{z}$ in $\mathfrak{n}$.
Since $\mathfrak{n}$ is non-singular, then all operators $J_Z$ are bijective for nontrivial $Z\in \mathfrak{z}$.
Indeed, if $J_Z(X)=0$ for some non-trivial $X\in \mathfrak{v}$, then
$0=(J_Z(X),Y)=([X,Y],Z)$ for all $Y\in \mathfrak{v}$. Hence, the image of the operator $\ad(X):\mathfrak{n} \rightarrow \mathfrak{z}$ is not whole $\mathfrak{z}$,
that is impossible. Therefore, the operator $J_Z:\mathfrak{v} \rightarrow \mathfrak{v}$ is bijective for any $Z\neq 0$. Since $J_Z$ is skew-symmetric, then
$n=\dim(\mathfrak{v})$ is even.

Let us consider the unit sphere $S=\{X\in \mathfrak{v}\>|\,(X,X)=1\}$ in $\mathfrak{v}$. Any $Z\in \mathfrak{z}$ determines a tangent vector fields on $S$ as follows:
$J_Z(X)$ is a tangent vector to $S$ at the point $X\in S$. Therefore,
the sphere $S$ admits $m$ linear independent tangent vector fields, where $m=\dim(\mathfrak{z})$.
It is known that $0\leq \dim (\mathfrak{z})=m < \rho(n)$, where $\rho$ is the function defined by
$\dim(\mathfrak{v})=n = (2a+1) \cdot 2^{\,4b+c}\mapsto \rho(n)=8b+2^{\,c}$, where $a,b,c \in \mathbb{N}$, $0\leq c \leq 3$, see e.~g. \cite[Theorem 8.2]{Hus}.
In particular, we get that $n$ even for $m=1$, $n=4k$ for $m\in\{2,3\}$, $n=8k$ for $m\in\{4,5,6,7\}$, where $k \in \mathbb{N}$.

Important examples of non-singular two-step nilpotent Lie algebra are so called $H$-type
algebras, which generalize Heisenberg algebras.

Let $\mathfrak{n}$ be a two-step nilpotent Lie algebra. We say that $\mathfrak{n}$ is an $H$-type
algebra if there exists an inner product $(\cdot,\cdot)$ such that $J_Z^2=-(Z,Z) \Id$ for every $Z \in \mathfrak{z}$.
We note that Heisenberg algebras are exactly $H$-type algebras with one-dimensional centers $\mathfrak{z}$.

For any $H$-type algebra, the orthogonal complement $\mathfrak{v}=\mathfrak{z}^{\perp}$ is a Cliffold module over the Clifford algebra
$C(\mathfrak{z})$ generated by $\mathfrak{z}$ and $1$ modulo relation $J_Z^2+(Z,Z)\cdot 1=0$, $Z\in \mathfrak{z}$.
Indeed, the linear map $J$ from $\mathfrak{z}$ to $\mathfrak{v}$ extends to a representation of $C(\mathfrak{z})$, and $\mathfrak{v}$ is
$C(\mathfrak{z})$-module. Moreover, every Clifford module arises in this way \cite{Kap81}.
The Clifford
modules are completely classified, see details e.~g. in \cite{Hus, Kap81}.
\smallskip

Some examples of geodesic orbit nilmanifolds are constructed using homogeneous fiber bundles
over compact symmetric spaces in \cite{Tam03}.
\smallskip

We have one useful isomorphism invariant for two-step Lie algebras $\mathfrak{n}=\mathfrak{z}\oplus\mathfrak{v}$ with even $n=\dim \mathfrak{v}$,
called {\it the Pfaffian form}, which is the projective equivalence class of the homogeneous
polynomial $f_{\mathfrak{n}}$ of degree $n/2$ in $m=\dim(\mathfrak{z})$ variables defined by
$$
\Bigl( f_{\mathfrak{n}}(Z) \Bigr)^2 = \det (J_Z), \qquad Z\in \mathfrak{z}.
$$
It is known that $\mathfrak{n}$ is non-singular if and only if $f_{\mathfrak{n}}(Z)$
is a positive polynomial, i.~e., $f_{\mathfrak{n}}(Z)>0$ for any non-zero $Z\in \mathfrak{z}$, see details in \cite{Sc67}.
\smallskip

\section{Some natural examples}\label{sec_3}

At first, we consider the following two $3$-parameter families of matrices from $\mathfrak{so}(4)$:
\begin{equation}\label{eq_leri_mult_1}
L(\beta_1,\beta_2,\beta_3)=\left(
\begin{array}{cccc}
0& -\beta_1& -\beta_2& -\beta_3\\
\beta_1& 0& -\beta_3& \beta_2\\
\beta_2& \beta_3& 0& -\beta_1\\
\beta_3& -\beta_2& \beta_1& 0\\
\end{array}%
\right), \quad \beta_1,\beta_2,\beta_3\in \mathbb{R},
\end{equation}
\begin{equation}\label{eq_leri_mult_2}
R(\gamma_1,\gamma_2,\gamma_3)=\left(
\begin{array}{cccc}
0& -\gamma_1& -\gamma_2& -\gamma_3\\
\gamma_1& 0& \gamma_3& -\gamma_2\\
\gamma_2& -\gamma_3& 0& \gamma_1\\
\gamma_3& \gamma_2& -\gamma_1& 0\\
\end{array}%
\right),  \quad \gamma_1,\gamma_2,\gamma_3\in \mathbb{R}.
\end{equation}

The following results are well known (and they are easy to prove):
Every matrix $U\in \mathfrak{so}(4)$ can be uniquely presented as
$L(\beta_1,\beta_2,\beta_3)+R(\gamma_1,\gamma_2,\gamma_3)$ for suitable $\beta_i$ and $\gamma_i$, $i=1,2,3$.
Moreover, $\left[L(\beta_1,\beta_2,\beta_3),R(\gamma_1,\gamma_2,\gamma_3)\right]=0$ for any values of $\beta_i$ and $\gamma_i$, $i=1,2,3$.
The matrices $L(\beta_1,\beta_2,\beta_3)$ for various values of $\beta_i$, $i=1,2,3$,
constitutes a Lie algebra isomorphic to $\mathfrak{so}(3)$. The same we can say about the matrices
$R(\gamma_1,\gamma_2,\gamma_3)$ for various values of $\gamma_i$, $i=1,2,3$.
Hence, we deal with the decomposition $\mathfrak{so}(4)=\mathfrak{so}(3) \oplus \mathfrak{so}(3)$
of $\mathfrak{so}(4)$ into the direct sum of its three-dimensional ideals (we may assume the first summand is determined by matrices of the form
$L(\beta_1,\beta_2,\beta_3)$, while the second summand is determined by matrices of the form $R(\gamma_1,\gamma_2,\gamma_3)$.

For a given $r>0$, we consider $S^3_r=\{X=(x_1,x_2,x_3,x_4)\in \mathbb{R}^4\,|\, x_1^2+x_2^2+x_3^2+x_4^2=r^2\}$, the sphere of radius $r$ with center at the origin.
Note that, the tangent plane to $S^3_r$ at the point $U\in S^3_r$ is naturally identified with $\{W(U)\,|\,W \in \mathfrak{so}(4)\}$.
The following result is well known,  but we consider an outline of its proof for completeness.

\begin{lemma}\label{le_leri_mult_1}
For any tangent vector $V$ to $S^3_r$ at given point $U\in S^3_r$, there is a triple of $\beta_1,\beta_2,\beta_3$,
as well a triple of $\gamma_1,\gamma_2,\gamma_3$, such that
$L(U)=R(U)=V$, where  $L=L(\beta_1,\beta_2,\beta_3)$ and $R=R(\gamma_1,\gamma_2,\gamma_3)$.
\end{lemma}

\begin{proof}
Note that the three-dimensional sphere $S^3$ topologically is the Lie group $Sp(1)$, the group of unit quaternions.
We have a natural action of $Sp(1)\times Sp(1)$ on $S^3=Sp(1)$ as follows: $(q_1,q_2): q \mapsto q_1\cdot q \cdot q_2^{-1}$.
In particular, we have a surjective homomorphism  $\psi:Sp(1)\times Sp(1) \mapsto SO(4)$ with the ineffective kernel $\mathbb{Z}_2=\{(1,1), (-1,-1)\}$.
The Lie algebras of the images of the first and the second multiples in $Sp(1)\times Sp(1)$ under $\psi$ are
the first and the second ideals in the Lie algebra $\mathfrak{so}(4)=\mathfrak{so}(3) \oplus \mathfrak{so}(3)$, that are determined by the
matrices $L(\beta_1,\beta_2,\beta_3)$ and $R(\gamma_1,\gamma_2,\gamma_3)$ respectively, see \eqref{eq_leri_mult_1} and  \eqref{eq_leri_mult_2}.

Both these images of $Sp(1)$ under $\psi$ act transitively on $S^3\subset \mathbb{R}^4$, see e.~g. \cite{MonSa43}. It implies the following observation:
If $U=(u_1,u_2,u_3,u_4)\in \mathbb{R}^4$ and $r=\bigl(u_1^2+u_2^2+u_3^2+u_4^2\bigr)^{1/2}$,
then for any $V$ in the tangent plane to the sphere $S^3_r$ at the point $U$,
there is an element $W$ in the chosen ideal  $\mathfrak{so}(3)$ such that $W(U)=V$.
\end{proof}
\smallskip

Lemma \ref{le_leri_mult_1} and the discussion before it imply the following observation: If $\mathcal{V}$ is a linear subspace
in one copy of $\mathfrak{so}(3)$ for the decomposition $\mathfrak{so}(4)=\mathfrak{so}(3) \oplus \mathfrak{so}(3)$, then the second copy
of $\mathfrak{so}(3)$, which will be denoted as $\mathcal{Z}$, centralizes $\mathcal{V}$. Moreover, $\mathcal{V}$
satisfies the transitive normalizer condition with respect to $\mathcal{Z}$, hence every inner product on $\mathcal{V}$
determines a GO-nilmanifold of dimension $\dim(\mathcal{V})+4$. Up to isomorphism, there is only one choice of $\mathcal{V}$ of dimension $1$, $2$, or $3$.
\smallskip

Let us consider a more general situation.
Suppose that $\mathcal{A}\oplus \mathcal{B}$ is a Lie subalgebra of $\mathfrak{so}(n)$ such that the corresponding Lie group $B\subset SO(n)$, where $\mathcal{B}=\Lie(B)$,
acts transitively on the unit sphere $S$ of $\mathbb{R}^n=\mathfrak{v}$.
All such Lie groups classified in \cite{MonSa43}. Note that for any $X \in S$ and any $U\in \mathfrak{so}(n)$, $U(X)$ is a tangent vector to $S$ at the point $X$.
Let us consider any linear subspace $\mathcal{V} \subset \mathcal{A}$. Then
$\mathcal{V}$ satisfies the transitive normalizer condition. Indeed, $\mathcal{B}$ is a subset of the normalizer $\mathcal{N}$ of $\mathcal{V}$ in $\mathfrak{so}(n)$.
If we fix $Z\in \mathcal{V}$ and $X \in S \subset \mathbb{R}^n$, then we can find $Y\in \mathcal{B}$ such that
$Y(X)=Z(X)$. It follows from the fact that $\Lie(\mathcal{B})$ acts transitively on the unit sphere $S$, hence, $\mathcal{B}(X)$ coincides with the tangent
space to $S$ at the point $X$. Moreover, by our assumptions, $[Y,Z]=0$. Hence, $\mathcal{V}$ satisfies the transitive normalizer condition.
By Proposition \ref{gonil2} we get a family of geodesic orbit nilmanifolds, corresponding to the metric Lie algebras
$(\mathcal{V}\rtimes \mathbb{R}^n, (\cdot,\cdot)_1+(\cdot,\cdot)_2)$, where
$(\cdot,\cdot)_1$ is any $\ad(\mathcal{B})$-invariant inner product on $\mathcal{V}$ and
$(\cdot,\cdot)_2$ is the standard inner product in $\mathbb{R}^n$. Since $[\mathcal{V},\mathcal{B}]=0$, $(\cdot,\cdot)_1$ is any invariant inner product on $\mathcal{V}$
(here we take $\mathcal{B}$ as $\mathcal{N}'$ in the notations of Remark \ref{re_sub_1}).
\smallskip

In particular, we know that $U(1)\cdot SU(n) \subset SO(2n)$ and $SU(n)$ acts transitively on the unit sphere $S$ of $\mathbb{R}^{2n}$.
Therefore, we can consider $\mathcal{V}=\mathcal{A}=\mathfrak{u}(1)$ and $\mathcal{B}=\mathfrak{su}(n)$.
The corresponding GO-nilmanifolds are the Heisenberg groups with suitable Riemannian metrics.
\smallskip

We know also that $Sp(1)\cdot Sp(n)\subset SO(4n)$ and $Sp(n)$ acts transitively on the unit sphere $S$ of $\mathbb{R}^{4n}$.
Hence, we can consider $\mathcal{V}=\mathcal{A}=\mathfrak{sp}(1)$ and $\mathcal{B}=\mathfrak{sp}(n)$. The corresponding GO-nilmanifold are the quaternionic Heisenberg groups
\cite{AFS15}.  In these two partial cases we get all possible naturally reductive $H$-type group, see \cite[Proposition 1]{Kap83}.
Note that the first examples of commutative spaces which are not weakly symmetric
(which provides an answer to Selberg's question about the existence of such examples \cite{S}) is modeled as the quaternionic Heisenberg group,
endowed with certain special naturally reductive metrics~\cite{Lau14}.
\smallskip

Instead of $Sp(1)$ we can take also $\mathcal{V}=U(1)\subset Sp(1)=\mathcal{A}$ and any two-dimensional subspace $\mathcal{V}\subset Sp(1)=\mathcal{A}$.
In the first case we again obtain the Heisenberg groups (only of dimension $4n+1$), in the second case we
get $H$-type groups of dimension $4n+2$, that are not naturally reductive, but are geodesic orbit, see \cite[Proposition 3]{Kap83}.
\smallskip

The complete classification of geodesic orbit $H$-type groups was obtained by C.~Riehm in \cite{Rie84}:
a given $H$-type group is geodesic orbit if and only if $m = 1, 2, 3$; or
$m = 5, 6$ and $n = 8$; or $m = 7$, $n =8, 16, 24$ and $\mathfrak{v}$ is an isotypic Clifford module (in this case it is equivalent to the following property:
if $Z_1,Z_2,\dots,Z_7$ is an orthonormal
basis of $\mathfrak{z}$,
the linear transformation $T: X \mapsto J_{Z_1}(J_{Z_2}(\cdots J_{Z_7}(X)\cdots))$ of $\mathfrak{v}$ is either $\Id$ or $-\Id$).
\medskip

The classification of weakly symmetric $H$-type group was obtained  in \cite{BRV}:
a given $H$-type group is weakly symmetric if and only if $m = 1, 2, 3$; or
$m = 5, 6, 7$ and $n = 8$; or $m = 7$, $n =16$ and $\mathfrak{v}$ is an isotypic Clifford module.
\medskip

There are many examples of non-singular two-step nilpotent Lie algebras that are not of $H$-type \cite{LauOs14, LT99}.
Useful information about automorphisms of this class of Lie algebras can be found in \cite{Saal96, KapTir}.

\section{GO-nilmanifolds of the centralizer type}\label{sec_4}

In this section, we discuss one special class of geodesic orbit nilmanifolds.
Nilmanifolds from this class have a number of special properties, so there is a chance for their classification (at least in small dimensions).

\begin{definition}\label{de_gonwrt}
We say that a Riemannian GO-nilmanifold $(N,g)$ is {\it of the centralizer type} if in the notation of
Proposition \ref{gonil2} and Remark \ref{re_sub_1},
$\mathcal{V}$
satisfies the transitive normalizer condition with respect to the centralizer $\mathcal{Z}$ of $\mathcal{V}$ in $\mathfrak{so}(n)$, i.~e.
for every $X\in \mathbb{R}^n$ and every $Y\in \mathcal{V}$  there is some $Z \in \mathcal{Z}$ such that
$[Y,Z]=0$ and $Y(X)=Z(X)$.
\end{definition}

\smallskip

It is clear that  the centralizer $\mathcal{Z}$ of $\mathcal{V}$ in $\mathfrak{so}(n)$
is a normal subgroup in the normalizer $\mathcal{N}$ of $\mathcal{V}$ in $\mathfrak{so}(n)$.
Moreover, the quotient group $\mathcal{N}/\mathcal{Z}$ has a faithful representation in $\mathcal{V}$.

\begin{lemma}\label{le_centy_1}
Let $(N,g)$ be a Riemannian GO-nilmanifold $(N,g)$ of the centralizer type,
$\mathcal{A}$ is  a Lie subalgebra in $\mathfrak{so}(n)$ generated by $\mathcal{V}$, $C(\mathcal{A})$ and $S(\mathcal{A})$
are the center and semisimple part of $\mathcal{A}$ respectively. Then the following
assertions hold:

{\rm 1)} $[\mathcal{A},\mathcal{Z}]=0$;

{\rm 2)} $C(\mathcal{A})=\mathcal{A} \bigcap \mathcal{Z} \subset C(\mathcal{Z})$, where $C(\mathcal{Z})$ the center of $\mathcal{Z}$;

{\rm 3)} If $U=U_1+U_2$ for any $U \in \mathcal{V}$, where $U_1\in C(\mathcal{A})$ and $U_2\in S(\mathcal{A})$,
then we have $C(\mathcal{A})=\{U_1\,|\, U\in \mathcal{V}\}$;

{\rm 4)} $\dim (C(\mathcal{A})) \leq \dim  \mathcal{V}$.
\end{lemma}

\begin{proof}
If $\mathcal{V}^{(1)}=[\mathcal{V},\mathcal{V}]$ and
$\mathcal{V}^{(i+1)}=[\mathcal{V}^{(i)},\mathcal{V}]$ for $i\geq 1$, then $[\mathcal{Z},\mathcal{V}^{(i)}]=0$ for any natural $i$ due to  $[\mathcal{Z},\mathcal{V}]=0$
and $[\mathcal{Z},\mathcal{V}^{(i+1)}]\subset [[\mathcal{Z},\mathcal{V}^{(i)}],\mathcal{V}]+[\mathcal{V}^{(i)},[\mathcal{Z},\mathcal{V}]]$ for $i\geq 1$.
Since $\mathcal{A}$ is generated by $\mathcal{V}$, then $[\mathcal{Z},\mathcal{A}]=0$, that proves the first assertion.

We know that $[\mathcal{A},\mathcal{Z}]=0$, therefore, if $U \in \mathcal{A} \bigcap \mathcal{Z}$, then  $U\in \mathcal{A}$
and $[U,\mathcal{Z}]=0$, hence, $U \in C(\mathcal{Z})$.
Moreover, if $U \in \mathcal{A} \bigcap \mathcal{Z}$, then $U\in \mathcal{Z}$ and $[U,\mathcal{A}]=0$ by the previous assertion. Since $U \in \mathcal{A}$, then
$U \in C(\mathcal{A})$. On the other hand, if $U \in C(\mathcal{A})\subset \mathcal{A}$, then
$[U,\mathcal{V}]\subset[U,\mathcal{A}]=0$, hence $U\in \mathcal{Z}$, the centralizer of
$\mathcal{V}$ in $\mathfrak{so}(n)$. Therefore, $U \in \mathcal{A} \bigcap \mathcal{Z}$. This proves the second assertion.

The third assertions follows from the facts that $\mathcal{A}$ is generated by $\mathcal{V}$ and $\mathcal{V}^{(i)}\in S(\mathcal{A})$ for all $i \geq 1$.
Since the map $U \mapsto U_1$ (see above) from $\mathcal{V}$ to $C(\mathcal{A})$ is linear, we have $\dim (C(\mathcal{A})) \leq \dim  \mathcal{V}$,
that proves the forth assertions.
\end{proof}

\begin{prop}\label{pr_centy_1}
Let $(N,g)$ be a Riemannian GO-nilmanifold $(N,g)$ of the centralizer type,
$\mathcal{A}$ is  a Lie subalgebra in $\mathfrak{so}(n)$ generated by $\mathcal{V}$, $C(\mathcal{A})$ and $S(\mathcal{A})$
are the center and semisimple part of $\mathcal{A}$ respectively, $\pi: \mathcal{V} \rightarrow S(\mathcal{A})$ be a linear projection along $C(\mathcal{A})$.
Then $\mathcal{Z}$ is the centralizer of $\mathcal{V}_s:=\pi(V) \subset S(\mathcal{A})$.
Moreover, $\mathcal{V}_s$ and $\mathcal{Z}$ determine a Riemannian GO-nilmanifold of the centralizer type.

On the other hand, if $\mathcal{V} \subset S(\mathcal{A})$  and $\psi: \mathcal{V} \rightarrow C(\mathcal{Z})$
is any linear map, where $C(\mathcal{Z})$ is the center of $\mathcal{Z}$, then $\mathcal{V}_{\psi}:=\{Z+\psi(Z)\,|\, Z \in \mathcal{V}\}$
and $\mathcal{Z}$ determine a Riemannian GO-nilmanifold of the centralizer type.
\end{prop}

\begin{proof}
It is clear that $S(\mathcal{A})$ is generated by $\mathcal{V}_s$ and $[\mathcal{Z},\mathcal{V}_s]\subset [\mathcal{Z},\mathcal{A}]=0$, whereas
$\mathcal{A} \bigcap \mathcal{Z}=C(\mathcal{A})=\{Z-\pi(Z)\,|\, Z \in \mathcal{V}\}$ by Lemma \ref{le_centy_1}.
Since $(N,g)$ is a Riemannian GO-nilmanifold $(N,g)$ of the centralizer type,
for every $Y\in \mathbb{R}^n$ and every $Z\in \mathcal{V}$  there is some $X \in \mathcal{Z}$ such that
$[X,Z]=0$ and $X(Y)=Z(Y)$.
If $Z=Z_1+Z_2$, where $Z_1\in C(\mathcal{A})$ and $Z_2=\psi(Z)\in \mathcal{V}_s \subset S(\mathcal{A})$.
Since $C(\mathcal{A})=\mathcal{A} \bigcap \mathcal{Z} \subset C(\mathcal{Z})$ we can consider $X'=X-Z_1\in \mathcal{Z}$.
It is easy to see that $[X',\pi(Z)]=[X-Z_1,Z_2]=0$ and $X'(Y)=X(Y)-Z_1(Y)=Z(Y)-Z_1(Y)=Z_2(Y)$.
This proves that $\mathcal{V}_s$ and $\mathcal{Z}$ determine a Riemannian GO-nilmanifold of the centralizer type.
Note that $\mathcal{V} \subset S(\mathcal{A})$ implies that $\pi$ is the identity mapping.

Now, if $\mathcal{V} \subset S(\mathcal{A})$, for any $Z\in  \mathcal{V}$ we consider $Z':=Z+\psi(Z)\in \mathcal{V}_{\psi}$.
Since $(N,g)$ is a Riemannian GO-nilmanifold, then for every $Y\in \mathbb{R}^n$ and every $Z\in \mathcal{V}$  there is some $X \in \mathcal{Z}$ such that
$[X,Z]=0$ and $X(Y)=Z(Y)$. Let us consider $X':=X+\psi(Z)$, then $[X',Z']=[X+\psi(Z), Z+\psi(Z)]=0$ and $X'(Y)=Z'(Y)$ for a given $Y\in \mathbb{R}^n$.
This proves that $\mathcal{V}_{\psi}$ and $\mathcal{Z}$ determine a Riemannian GO-nilmanifold of the centralizer type.
\end{proof}

\begin{remark}
Instead of the centralizer $\mathcal{Z}$ in Proposition \ref{pr_centy_1}, one can take any subalgebra $\mathcal{Z}'$ of
$\mathcal{Z}$ with the following property: For every $Y\in \mathbb{R}^n$ and every $Z\in \mathcal{V}$  there is some $X \in \mathcal{Z}'$ such that
$[X,Z]=0$ and $X(Y)=Z(Y)$. The above proposition is interesting only in the case when the center of $\mathcal{Z}'$ is not trivial.
\end{remark}

\begin{remark}
Proposition \ref{pr_centy_1} implies the following method to construct GO-nilmanifolds of the centralizer type.
Suppose that there is a direct sum of Lie algebras $\mathcal{A}\oplus \mathcal{B}$ in  $\mathfrak{so}(n)$, such that
$\mathcal{A}$ is a semisimple Lie algebra and the Lie subgroups $G_1 \subset G_2 \subset SO(n)$ with the Lie algebras
$\mathcal{B}\subset \mathcal{A}\oplus \mathcal{B}\subset \mathfrak{so}(n)$
are such that any orbit of $G_2$  in $\mathbb{R}^n$ is also an orbit of $G_1$.
Then any linear subspace $\mathcal{V} \subset \mathcal{A}$ satisfies
the transitive normalizer condition with respect to the centralizer $\mathcal{Z}$ of $\mathcal{V}$ in $\mathfrak{so}(n)$.
A partial case of this construction (when $G_1$ acts transitively on the unit sphere in $\mathbb{R}^n$) is considered in Section 3.
\end{remark}

It is clear that the property to be of {\it the centralizer type} is rather restrictive for GO-nilmanifolds. On the other hand, we have

\begin{prop}\label{pr_centy_2}
If a Riemannian GO-nilmanifold $(N,g)$ is such that
$\dim(\mathcal{V})=2$ in the notation of
Proposition \ref{gonil2}, then it is of the centralizer type.
\end{prop}

\begin{proof}
Since $(N,g)$ is a GO-nilmanifold, then (by Proposition \ref{gonil2})
for every $Y\in \mathbb{R}^n$ and every $Z\in \mathcal{V}$  there is some $X \in \mathcal{N}$ such that
$[X,Z]=0$ and $X(Y)=Z(Y)$. If $Z$ is non-trivial, then the inclusion $[X,\mathcal{V}]\subset \mathcal{V}$ and the equality $[X,Z]=0$
imply, that $[X,\mathcal{V}]=0$, hence,  $X \in \mathcal{Z}$, where  $\mathcal{Z}$ is the centralizer of $\mathcal{V}$ in $\mathfrak{so}(n)$.
Indeed, the orthogonal compliment to $\mathbb{R}\cdot Z$ in $\mathcal{V}$ is $\ad(X)$-invariant and $1$-dimensional,
hence $X$ acts trivially on it.
\end{proof}
\smallskip

Finally, we consider some
restrictions on spectral properties of elements of $\mathcal{V}$ and $\mathcal{N}$ (or $\mathcal{Z}$) as in Proposition \ref{gonil2}.
These properties will be useful in classifying low-dimensional GO-nilmanifolds.

\begin{prop}\label{pr_centy_3}
Suppose that $U, V \in \mathfrak{so}(n)$ with $[U,V]=0$ and consider the subspace
$$
L_{U,V}=\{Z\in \mathbb{R}^n \,|\, U(Z)=V(Z)\}.
$$
Then $L_{U,V}$ is an an invariant subspace both for $U$ and $V$.
Moreover, if $Z\in L_{U,V}$ and $Z=Z_0+Z_1+\cdots + Z_p$, where  $Z_i$ is an eigenvector of $U^2$ with eigenvalues $\alpha_i^2$, where
$0=|\alpha_0|<|\alpha_1|<\cdots < |\alpha_p|$,
then $Z_i \in  L_{U,V}$ for all $i=0,1,\dots, p$.
In this case, every $Z_i$ is also an eigenvector of $V^2$ with the eigenvalue $\alpha_i^2$ and $U(Z_i)=V(Z_i)$ is an eigenvector both for $V^2$ and $U^2$
with the same eigenvalue $\alpha_i^2$.
\end{prop}

\begin{proof}
Take any $Z\in L_{U,V}$. We have to show that $U(Z) \in  L_{U,V}$ and $V(Z) \in  L_{U,V}$.
By the definition of $L_{U,V}$, we have $U(Z)=V(Z)=:Z'$. Since $[U,V]=0$, we have $0=[U,V](Z)=U(V(Z))-V(U(Z))=U(Z')-V(Z')$, hence, $Z' \in L_{U,V}$.
This proves the first assertion.

It is clear that the operator $L:=U^2=U\circ U$ is symmetric on $\mathbb{R}^n$.
Note, that $\alpha^2_i \in \mathbb{R}$ for all $i$ and $\alpha^2_i<0$ for $i \geq 1$ (since $U$ is skew-symmetric).
Put $F_0:=Z$ and $F_{i+1}:=L(F_i)$ for $i=1,\dots, p-1$.
By the first assertion, we get $F_i \in L_{U,V}$ for $i=0,1,\dots, p$. On the other hand, we have $F_i=(-1)^i \sum_{i=0}^p |\alpha_i|^{2i} Z_i$, $i=0,1,\dots, p$.
We get a system of $p+1$ linear equations with respect to $Z_0, Z_1, \dots, Z_p$ with non-zero determinant (it is the Vandermonde determinant),
hence, solving it, we represent every $Z_i$ as a linear combination of the vectors $F_j$, $j=0,1,\dots,p$. Since all this vectors are from
$L_{U,V}$, then $Z_i \in  L_{U,V}$ for all $i=0,1,\dots, p$.

Since $U(Z_i)=V(Z_i)$, $U(U(Z_i))=V(V(Z_i))=\alpha_i^2 Z_i$ and $U(U(U(Z_i)))=\alpha_i^2 U(Z_i)=\alpha_i^2 V(Z_i)=V(V(V(Z_i)))$,
then $Z_i$ is also an eigenvector of $V^2$ with the eigenvalue $\alpha_i^2$,
$U(Z_i)=V(Z_i)$ is an eigenvector for $V^2$ and $U^2$
with the eigenvalue $\alpha_i^2$.
The proposition is proved.
\end{proof}

\begin{lemma}\label{le_centy_3}
Suppose that $U, V, W \in \mathfrak{so}(n)$ and $[U,V]=[U,W]=0$.
If the matrix $U$ has no non-zero eigenvalue of multiplicity $\geq 2$ and $0$ is an eigenvalue of $U$ of multiplicity $\leq 2$, then $[V,W]=0$.
\end{lemma}

\begin{proof}
Since $[U,V]=0$ then there is an orthonormal basis $(e_i)_{1\leq i \leq n}$ in $\mathbb{R}^n$,
in which
$$
U=\diag(U_1,\cdots, U_p,0,\cdots,0),\quad V=\diag(V_1,\cdots, V_p,0,\cdots,0),
$$
for some $p\leq n/2$, where $U_i$ and $V_i$ are skew-symmetric ($2\times2$)-matrices, $p\leq n/2$.
see \cite[Chapter 9, \S 15, Theorem 12]{Gant}.
Since $[U,W]=0$ then there is an orthonormal basis $(f_i)_{1\leq i \leq n}$ in $\mathbb{R}^n$,
in which
$$
U=\diag(U'_1,\cdots, U'_q,0,\cdots,0),\quad W=\diag(W_1,\cdots, W_q,0,\cdots,0),
$$
for some $q\leq n/2$, where $U'_i$ and $W_i$ are skew-symmetric ($2\times2$)-matrices.
Since the matrix $U$ has no non-zero eigenvalue of multiplicity $\geq 2$ and $0$ is an eigenvalue of $U$ of multiplicity $\leq 2$,
then after some permutation of vectors in the basis $(f_i)_{1\leq i \leq n}$ we get that $U'_i=U_i$ for all $1\leq i \leq \min\{p,q\}$.
These ($2\times2$)-matrices corresponds to pairs of complex conjugate eigenvalues of $U$, while all other eigenvalues are zero
(the eigenvalue $0$ is of multiplicity $0$, $1$ or $2$). Therefore,
any $2$-dimensional space $\Lin (e_{2k+1}, e_{2k+2})$, $k=0,\dots,\min\{p,q\}-1$, is eigenspace both for $V$ and $W$. The orthogonal compliment in $\mathbb{R}^n$
to the the direct sum of all such spaces is also an eigenspace both for $V$ and $W$, and has dimension $\leq 2$ (since $0$ is an eigenvalue of $U$ of multiplicity $\leq 2$).
It is clear that in all above 2-dimensional and (possible) $1$-dimensional eigenspaces
$V$ commutes with $W$. Hence, $[V,W]=0$.

Another proof of this lemma can be obtained from the structure of adjoint orbits of the Lie algebra $\mathfrak{so}(n)$, see e.~g. \cite[Chapter 8, 8.113--8.115]{Bes}.
\end{proof}

\section{The proofs of the main results}\label{sec_5}

According to \cite[Corollary 4.9]{LT99},
there are two non-isomorphic $10$-dimensional non-singular Lie algebras with  Pfaffian form $(x^2+y^2)^2$
and a $1$-parameter family (with respect to a real parameter $t>1$) of pairwise non-isomorphic $10$-dimensional non-singular Lie algebras with the Pfaffian forms
$(x^2+y^2)(t^2\cdot x^2+y^2)$, $t>1$. Here we assume that $Z=x\cdot Z_1 +y \cdot Z_2$, where the vectors $Z_1, Z_2$ form an orthonormal basis in $\mathfrak{z}$.
We denote the above $10$-dimensional non-singular Lie algebra by $\mathfrak{n}_{10,1}$, $\mathfrak{n}_{10,2}$, and $\mathfrak{n}_{10,t}$, $t>1$, respectively.

The first of theses $10$-dimensional non-singular Lie algebras are defined by the operator
$J_Z=\diag (C_1,C_1)$, while the last family of Lie algebras is defined by the operators
$J_Z=\diag(C_1,C_2)$, where

\begin{equation}\label{eq_c1c2}
C_1=
\left(%
\begin{array}{cccc}
  0 & 0 & -x & y \\
  0 & 0 & y & x \\
  x & -y & 0 & 0 \\
  -y & -x & 0 & 0 \\
\end{array}%
\right), \quad
C_2=
\left(%
\begin{array}{cccc}
  0 & 0 & -t\cdot x & y \\
  0 & 0 & y & t\cdot x \\
  t\cdot x & -y & 0 & 0 \\
  -y & -t\cdot x & 0 & 0 \\
\end{array}%
\right).
\end{equation}

The second $10$-dimensional non-singular Lie algebras is defined by the operator
$J_Z=\left(%
\begin{array}{cc}
  0 & A \\
  -A^{\prime} & 0\\
\end{array}%
\right)$
on $\mathbb{R}^8$ (the symbol $A^{\prime}$ means the transpose matrix $A$), where
$$
A=\left(%
\begin{array}{cccc}
  0 & 0 & -x & y \\
  0 & 0 & y & x \\
  -x & y & 0 & x \\
  y & x & x & 0 \\
\end{array}%
\right).
$$

We see that the first Lie algebra is an $H$-type algebra and could be considered as member of the latter family of Lie algebras for $t=1$.
It is clear that $\mathcal{V}=\{J_{Z}\,|\, Z=x\cdot Z_1 +y \cdot Z_2 \in \mathfrak{z}\}$ is two-dimensional linear subspace in
$\mathfrak{so}(8)=\mathfrak{so}(\mathfrak{v})$.
These constructions are very useful in the proof of our first main result.
\smallskip

\begin{proof}[Proof of Theorem \ref{tm_1}]
For any $t\geq 1$ we consider $N_t$, a connected and simply connected nilpotent Lie group with the Lie algebra $\mathfrak{n}_{10,t}$, defined above.
Recall that for $t=1$, we obtain an $H$-type Lie group.

The non-isomorphism
within the family  $\mathfrak{n}_{10,t}$, $t>1$, follows by using the fact that the hyperbolic distance between $\sqrt{-1}$
and $t\cdot \sqrt{-1}$ on the complex plane strictly increases with $t$ (see details in \cite[Corollary 4.9]{LT99}),
or alternatively, from the non-equivalence between
their Pfaffian forms.

Now, consider a linear subspace $\mathcal{V}_t$ in $\mathfrak{so}(8)$, that consists of matrices of type
$J_Z=\diag(C_1,C_2)$, where
$C_1$ and $C_2$ are defined in
\eqref{eq_c1c2}.
The centralizer $\mathcal{Z}$ of $\mathcal{V}_t$ in $\mathfrak{so}(8)$ is a $6$-parameter
Lie subalgebra of $\mathfrak{so}(8)$, which consists of the matrices of the form
$Y=Y(\alpha_1,\alpha_2,\alpha_3,\alpha_4,\alpha_5,\alpha_6)=\diag(D_1,D_2)$, where
\begin{equation}\label{eq_10dim_1}
D_1=\left(
\begin{array}{cccc}
0 & -\alpha_1 & -\alpha_2& -\alpha_3\\
\alpha_1 & 0&  \alpha_3& -\alpha_2\\
\alpha_2& -\alpha_3& 0& \alpha_1\\
\alpha_3&\alpha_2&-\alpha_1&0\\
\end{array}%
\right), \quad
D_2=\left(
\begin{array}{cccc}
0& -\alpha_4& -\alpha_5& -\alpha_6\\
\alpha_4& 0& \alpha_6& -\alpha_5\\
\alpha_5& -\alpha_6& 0& \alpha_4\\
\alpha_6& \alpha_5& -\alpha_4& 0\\
\end{array}%
\right).
\end{equation}

Let us prove that $\mathcal{V}_t$ satisfies the transitive normalizer condition.

For any $X=(x_1,x_2,x_3,x_4,x_5,x_6,x_7,x_8)\in \mathbb{R}^8$
and any $Z \in \mathcal{V}_t$, it suffices to find a suitable $Y\in \mathcal{Z}$ such that $Y(X)=Z(X)$ (we know that $[Y,Z]=0$).
The latter equation is a system of $8$ linear equations with respect to variables $\alpha_i$, $1\leq i \leq 6$.
Nevertheless, this system has a solution for every $t,x,y$, and $x_i$, $1\leq i\leq 8$. It can be checked by straightforward computation.
We present the corresponding solutions in explicit form.

If $x_1^2+x_2^2+x_3^2+x_4^2 = 0$, then we can take any $\alpha_1, \alpha_2, \alpha_3 \in \mathbb{R}$.
If $x_1^2+x_2^2+x_3^2+x_4^2 \neq 0$, then we have
\begin{eqnarray*}
\alpha_1&=& \frac{2x(x_1x_4+x_2x_3) +2 y(x_1x_3-x_2x_4)}{x_1^2+x_2^2+x_3^2+x_4^2}, \\
\alpha_2&=& \frac{x(x_1^2 - x_2^2 +x_3^2 - x_4^2)-2y(x_1x_2+x_3x_4)}{x_1^2+x_2^2+x_3^2+x_4^2}, \\
\alpha_3&=& \frac{-2x(x_1x_2-x_3x_4)-y(x_1^2 - x_2^2 -x_3^2 + x_4^2)}{x_1^2+x_2^2+x_3^2+x_4^2}.
\end{eqnarray*}

Analogously, if $x_5=x_6=x_7=x_8=0$, then we can take any $\alpha_4, \alpha_5, \alpha_6 \in \mathbb{R}$.
In the case $x_5^2+x_6^2+x_7^2+x_8^2 \neq 0$, we have
\begin{eqnarray*}
\alpha_4&=& \frac{2tx(x_5x_8+x_6x_7) + 2y(x_5x_7-x_6x_8)}{x_5^2+x_6^2+x_7^2+x_8^2}, \\
\alpha_5&=& \frac{tx(x_5^2 - x_6^2 +x_7^2 - x_8^2) - 2y(x_5x_6+x_7x_8)}{x_5^2+x_6^2+x_7^2+x_8^2}, \\
\alpha_6&=& \frac{-2tx(x_5x_6+x_7x_8) - y(x_5^2-x_6^2-x_7^2+x_8^2)}{x_5^2+x_6^2+x_7^2+x_8^2}.
\end{eqnarray*}

According to Remark \ref{re_sub_1}, we can take $\mathcal{Z}$ as $\mathcal{N}'\subset \mathcal{N}$ in order to apply Proposition
\ref{gonil2}.
Since $[\mathcal{Z},\mathcal{V}_t]=0$ then any inner product on $\mathcal{V}_t$ is $\ad(\mathcal{Z})$-invariant.
Hence, there is a 3-parameter family of suitable inner product on $\mathcal{V}_t$ (since $\dim (\mathcal{V}_t)=2$).
Therefore, the theorem is completely proved.
\end{proof}
\medskip

Let us consider a more conceptual approach to proving of Theorem \ref{tm_1}.
Recall that the matrices
$\mathcal{V}_t$ have the form $\diag(C_1,C_2)$,   where the matrices $C_1,C_2$ are from \eqref{eq_c1c2}.
Moreover, $Y=Y(\alpha_1,\alpha_2,\alpha_3,\alpha_4,\alpha_5,\alpha_6)=\diag(D_1,D_2)$, where $D_1,D_2$ are from \eqref{eq_10dim_1}.

We see that the matrices $C_1$ and $C_2$ are from the first summand $\mathfrak{so}(3)$ in the decomposition
$\mathfrak{so}(4)=\mathfrak{so}(3) \oplus \mathfrak{so}(3)$ as in the proof of Lemma \ref{le_leri_mult_1}, while the matrices $D_1$ and $D_2$ are from the second one.
Hence, we see that $[\mathcal{V}_t, \mathcal{Z}]=0$.
Obviously, the set of all possible matrices $D_1$ (as well as the set of all possible matrices $D_2$) coincide with the set of matrices
$R(\gamma_1,\gamma_2,\gamma_3)$, that form a the second summand $\mathfrak{so}(3)$
in the above decomposition of $\mathfrak{so}(4)$.

Now we want to prove that for any $X\in \mathbb{R}^8$ and any $Y=(C_1,C_2)\in \mathcal{V}_t$ we can find $Z=(D_1,D_2)\in \mathcal{Z}$
such that $Z(X)=Y(X)$.
The main observation here is that the problem naturally  breaks down into two simpler problems.
Indeed, if $X=(X_1,X_2)$, where $X_1,X_2\in \mathbb{R}^4$, then $Z(X)=Y(X)$ is equivalent to two equalities: $D_1(X_1)=C_1(X_1)$ and $D_2(X_2)=C_2(X_2)$.

As we explained in the proof of Lemma \ref{le_leri_mult_1}, the vector $C_1(X_1)$ is in the tangent plane to the sphere  of radius $\sqrt{(X_1,X_1)}$  in $\mathbb{R}^4$
(centered at the origin),
and there is a matrix $D_1=R(\gamma_1,\gamma_2,\gamma_3)$ such that $D_1(X_1)=C_1(X_1)$. By the same reason, there is a matrix
$D_2=R(\gamma_1,\gamma_2,\gamma_3)$ such that $D_2(X_2)=C_2(X_2)$.
This proves the main assertions of Theorem \ref{tm_1} without explicit formulas.
\medskip

The above arguments lead also to the proof of our second main result.
\smallskip

\begin{proof}[Proof of Theorem \ref{tm_2}]
Let us define a family of two-dimensional subspaces in $\mathfrak{so}(4k+4)$ as follows. At first we fix pairwise distinct real numbers
$1=t_0<t_1<t_2<\cdots <t_k$ and define the matrices
$$
C_j=
\left(%
\begin{array}{cccc}
  0 & 0 & -t_j\cdot x & y \\
  0 & 0 & y & t_j\cdot x \\
  t_j\cdot x & -y & 0 & 0 \\
  -y & -t_j\cdot x & 0 & 0 \\
\end{array}%
\right), \qquad x,y \in \mathbb{R}, \quad 0\leq j \leq k.
$$
Now, we define
$$
\mathcal{V}_{t_1,t_2,\dots,t_k}=\diag(C_0,C_1,C_2,\cdots,C_k).
$$
It is clear that $\mathcal{V}_{t_1,t_2,\dots,t_k}$ forms a two-dimensional subspace in $\mathfrak{so}(4k+4)$ ($x,y \in \mathbb{R}$ are arbitrary)
for any fixed $k$-tuple  $\{t_i$\}, $1\leq i \leq k$.

The Pfaffian forms for the corresponding nilpotent Lie algebras $\mathfrak{n}_{\,t_1,t_2,\dots,t_k}$ with operators
$J_Z=\mathcal{V}_{t_1,t_2,\dots,t_k}\subset \mathfrak{so}(4k+4)$
are
$$
(x^2+y^2)(t_1^2\cdot x^2+y^2)\times \cdots \times (t_k^2\cdot x^2+y^2)=\prod_{j=0}^k (t_i^2\cdot x^2+y^2).
$$
Hence, for distinct $k$-tuples of $\{t_j\}$, $j=1,\dots,k$, we get non-isomorphic non-singular two-step nilpotent Lie algebra $\mathfrak{n}_{\,t_1,t_2,\dots,t_k}$,
see \cite{Sc67}.

Now, we note that the normalizer of $\mathcal{V}_{t_1,t_2,\dots,t_k}$ in $\mathfrak{so}(4k+4)$ contains
a subalgebra (that is the centralizer of $\mathcal{V}_{t_1,t_2,\dots,t_k}$ in $\mathfrak{so}(4k+4)$)
$$
\diag (D,D,\dots, D)\subset \mathfrak{so}(4k+4),
$$
where
$$
D=D(\alpha_1,\alpha_2,\alpha_3)=\left(
\begin{array}{cccc}
0 & -\alpha_1 & -\alpha_2& -\alpha_3\\
\alpha_1 & 0&  \alpha_3& -\alpha_2\\
\alpha_2& -\alpha_3& 0& \alpha_1\\
\alpha_3&\alpha_2&-\alpha_1&0\\
\end{array}%
\right), \quad \alpha_1, \alpha_2, \alpha_3 \in \mathbb{R}.
$$

Taking any $X\in \mathbb{R}^{4k+4}$, that we can represent it as $X=(X_0, X_1,\dots, X_k)$, where $X_j \in \mathbb{R}^4$, $j=0,1,2,\dots,k$.
Now, take any $Y=Y^{x,y} \in \mathcal{V}_{t_1,t_2,\dots,t_k}$ (we should fix $x$ and $y$ for this goal).
We consider also $C_j=C^{x,y}_j$ with the same fixed  $x$ and $y$ for all $j=0,1,\dots,k$,
i.~e. $Y^{x,y}=\diag(C^{x,y}_0,C^{x,y}_1,C^{x,y}_2,\cdots,C^{x,y}_k)$.

Further, for any $0\leq j\leq k$, we can find $\alpha_1^j,\alpha_2^j,\alpha_3^j$ such that
$D^j:=D(\alpha_1^j,\alpha_2^j,\alpha_3^j) (X_j)= C^{x,y}_j(X_j)$.
Therefore, $D^{X,x,y}:=\diag(D^0,D^1,\cdots,D^k)\in \mathcal{Z}$ is such that $D^{X,x,y}(X)=Y^{x,y}(X)$.

According to Remark \ref{re_sub_1}, we can take $\mathcal{Z}$ as $\mathcal{N}'\subset \mathcal{N}$ in order to apply Proposition
\ref{gonil2}.
Since $[\mathcal{Z},\mathcal{V}_{t_1,t_2,\dots,t_k}]=0$ then any inner product on $\mathcal{V}_{t_1,t_2,\dots,t_k}$ is $\ad(\mathcal{Z})$-invariant.
Hence, there is a 3-parameter family of suitable inner product on $\mathcal{V}_{t_1,t_2,\dots,t_k}$ (since $\dim (\mathcal{V}_{t_1,t_2,\dots,t_k})=2$).
Therefore, the theorem is completely proved.
\end{proof}
\bigskip

Note that  GO-nilmanifolds, constructed in the proofs of Theorems \ref{tm_1} and \ref{tm_2}, are non-singular and the Lie algebra $\mathcal{Z}\subset \mathcal{N}$
does not acts on $\mathfrak{v}=\mathbb{R}^{n}=\mathfrak{z}^{\perp}$ irreducibly. All irreducible submodules in $\mathfrak{v}$
are $4$-dimensional. This allows to define some GO-nilmanifolds of smaller dimensions in the spirit of \cite[Proposition 3.2]{GorNik2018}.
On the other hand, we get no new example of GO-nilmanifolds in this direction.

\vspace{5mm}

\bibliographystyle{amsunsrt}

\vspace{10mm}

\end{document}